\documentclass[11pt,a4paper]{article}

\usepackage{epsf,epsfig,amsfonts,amsgen,amsmath,amstext,amsbsy,amsopn,amsthm}
\usepackage{amsmath,times,mathptmx}
\usepackage{amsfonts,amsthm,amssymb}
\usepackage{amsfonts}
\usepackage{graphics}
\usepackage{latexsym,bm}
\usepackage{amsfonts,amsthm,amssymb,bbding}
\usepackage{indentfirst}
\usepackage{graphicx}
\usepackage{color}
\usepackage[colorlinks=true,anchorcolor=blue,filecolor=blue,linkcolor=blue,urlcolor=blue,citecolor=blue]{hyperref}
\usepackage{float}
\usepackage{tikz}

\evensidemargin=\oddsidemargin\topmargin=-1.5cm

\newtheorem{thm}{Theorem}[section]

\newtheorem{lem}{Lemma}[section]

\newtheorem{claim}{Claim}[section]
\newtheorem{definition}{Definition}[section]

\addtocounter{section}{0}

\begin{document}
\title{Clique supersaturation under a chromatic constraint below the Tur\'{a}n threshold
\footnote{Supported by the National Natural Science Foundation of China (Nos. 12471331 and 12501471).}}

\author{{\bf Benju Wang$^{a,b}$},
{\bf Longfei Fang$^{b}$}, {\bf Jinlong Shu$^{c}$}\thanks{Corresponding author. E-mail address: jlshu@shnu.edu.cn (J. Shu)}
\\
\small $^{a}$ School of Mathematics and Science, Shanghai Normal University, Shanghai 200234, China\\
\small $^{b}$ School of Mathematics and Finance, Chuzhou University, Chuzhou, Anhui 239012, China\\
\small $^{c}$ School of Finance and Business, Shanghai Normal University, Shanghai 200234, China}

\date{}
\maketitle
{\flushleft\large\bf \centerline{Abstract}}
A central theme in extremal graph theory is the supersaturation problem, which
investigates the minimum number of copies of a target subgraph forced by
prescribed edge conditions. This line of research goes back to Rademacher and
Erd\H{o}s for triangles, and was later extended to cliques by Lov\'asz and
Simonovits in the regime above the Tur\'an threshold. Mubayi further extended
this theory to color-critical graphs. Below the Tur\'an threshold, a closely
related existence-threshold phenomenon arises in the non-$p$-partite setting:
a classical result of Brouwer shows that, for $n\ge 2p+1$, every $n$-vertex
non-$p$-partite $K_{p+1}$-free graph has at most
$e(T_{n,p})-\lfloor n/p\rfloor+1$ edges.

Motivated by this threshold, we investigate a sharp clique-counting problem
below the Tur\'an threshold under the non-$p$-partite assumption. Let $p\ge 2$
and $s\ge 1$ be fixed integers. Let $Y_{n,p,s}$ be the graph obtained from
$T_{n,p}$ by adding an edge inside a largest part and deleting all but $s$ of
the edges from one endpoint of this new edge to a smallest part. Then
$e(Y_{n,p,s})=e(T_{n,p})-\lfloor n/p\rfloor+s+1$. We prove that, for all
sufficiently large $n$, every $n$-vertex non-$p$-partite graph $G$ with
$e(G)\ge e(Y_{n,p,s})$ contains at least as many copies of $K_{p+1}$ as
$Y_{n,p,s}$ does. The bound is sharp, as it is attained by the construction
$Y_{n,p,s}$. Thus our result provides the exact clique-counting analogue of
Brouwer's threshold for non-$p$-partite $K_{p+1}$-free graphs.

\begin{flushleft}
\end{flushleft}

\section{Introduction}
All graphs considered in this paper are finite and simple.
For a graph $G$, we denote by $|V(G)|$ (or simply $|G|$) and $e(G)$ its number of vertices and edges, respectively.
For an integer $p\ge2$, let $\#K_{p+1}(G)$ denote the number of copies of $K_{p+1}$ in $G$.
For positive integers $n$ and $p$,
the Tur\'an graph $T_{n,p}$ is the complete $p$-partite graph on $n$ vertices with the maximum number of edges.
The \emph{Tur\'{a}n number} $\operatorname{ex}(n,H)$ is the maximum number of edges in an $n$-vertex $H$-free graph.
Tur\'{a}n's theorem \cite{Turan} states that $\operatorname{ex}(n,K_{p+1})=e(T_{n,p})$,
with $T_{n,p}$ being the unique extremal graph;
the case $p=2$ corresponds to Mantel's theorem~\cite{Mantel1907}.
Equivalently, every $n$-vertex graph $G$ with $e(G)>e(T_{n,p})$
must contain at least one copy of $K_{p+1}$.

A natural quantitative refinement of Tur\'{a}n's theorem is to determine the minimum possible number of copies of $K_{p+1}$ under prescribed edge or structural conditions.
This belongs to the classical Erd\H{o}s--Rademacher-type supersaturation problem.
More generally, for a graph $F$,
let $h_F(n,q)$ denote the minimum number of copies of $F$ in an $n$-vertex graph with $\operatorname{ex}(n,F)+q$ edges.
Another closely related quantity is $t_F(n,q)$,
defined as the minimum number of copies of $F$ obtained by adding $q$ edges to an $n$-vertex extremal $F$-free graph.
A central question in this area is to determine when $h_F(n,q)=t_F(n,q)$.
For cliques, this line of research goes back to Rademacher's theorem,
which states that every $n$-vertex graph with $e(T_{n,2})+1$ edges contains at least $\lfloor \frac{n}{2}\rfloor$ triangles (see \cite{Erdos19627}).
Erd\H{o}s \cite{Erdos19626,Erdos19627} later extended this result by proving that $h_{K_3}(n,q)=t_{K_3}(n,q)$ for a linear range $q<cn$.
Lov\'{a}sz and Simonovits \cite{LM1975,LM1983} developed the corresponding theory for general cliques,
proving that $h_{K_r}(n,q)=t_{K_r}(n,q)$ when $q=o(n^2)$.
%studying the equality $h_{K_r}(n,q)=t_{K_r}(n,q)$ and solving the clique supersaturation problem in the range $q=o(n^2)$.
Further developments on the minimum number of cliques in graphs with given order and size can be found in \cite{Fisher1989,LO,Nikiforov2011,PA2017}.
Notably, Reiher \cite{Reiher2016} proved the clique density theorem,
determining the asymptotic minimum density of $(p+1)$-cliques in graphs with a fixed edge density.
These results show that the extremal problem
$$
\min\{\#K_{p+1}(G): |V(G)|=n,\ e(G)=m\}
$$
is a central quantitative counterpart of Tur\'{a}n-type extremal theory.

Another important direction concerns color-critical graphs.
Recall that a graph $F$ is color-critical with chromatic number $p+1$ if $\chi(F)=p+1$ and there exists an edge $e\in E(F)$ such that $\chi(F-e)=p$.
Color-critical graphs play a fundamental role in extremal graph theory.
Simonovits \cite{Simonovits1966} proved that, for every color-critical graph $F$ with $\chi(F)=p+1$,
the Tur\'{a}n graph $T_{n,p}$ is the unique extremal $F$-free graph for all sufficiently large $n$.
Denote by $c(n,F)$ the minimum number of copies of $F$ in the graph obtained from $T_{n,p}$ by adding one edge.
Mubayi \cite{Mubayi2010} obtained the following supersaturation result for color-critical graphs.

\begin{thm}[Mubayi \cite{Mubayi2010}]\label{thm-Mubayi}
Let $p\geq 2$ and $F$ be a color-critical graph
with $\chi (F)=p+1$. There exists $\delta_F>0$ such that
if $n$ is sufficiently large, $1\le q < \delta_F n$, and
$G$ is an $n$-vertex graph with
\begin{equation*} \label{eq-Mub}
e(G)\ge e(T_{n,p}) +q,
\end{equation*}
then $G$ contains at least $q\cdot c(n,F) $ copies of $F$.
\end{thm}

Pikhurko and Yilma \cite{PY2017} strengthened Mubayi's result by proving that,
for every color-critical graph $F$, if $q$ is sufficiently small compared to $n$,
then $h_F(n,q)=t_F(n,q)$,
thereby yielding the exact minimum rather than just a lower bound.
More recently, Ma and Yuan \cite{MaYuan2025} studied supersaturation beyond color-critical graphs and showed that,
for general non-bipartite graphs, the relation between $h_F(n,q)$, $t_F(n,q)$, and $q\cdot c(n,F)$ can be considerably more subtle.

The results above concern graphs whose number of edges exceeds the Tur\'{a}n number.
In this paper, we consider a different  regime.
We study graphs whose number of edges may lie below $e(T_{n,p})$ by a linear term,
but which are required to be non-$p$-partite.
The chromatic condition is essential in this range.
In fact, a $p$-partite graph contains no copy of $K_{p+1}$, even if its number of edges is very close to $e(T_{n,p})$.
Hence an edge condition alone cannot force the existence of $K_{p+1}$ below the Tur\'{a}n threshold.

This perspective is closely related to strong Tur\'an stability.
The classical Erd\H{o}s--Simonovits stability theorem~\cite{Erdos-1967,Erdos-1968,Simonovits1966} asserts that every $K_{p+1}$-free graph
with at least $e(T_{n,p})-o(n^2)$ edges is structurally close to $T_{n,p}$.
A more refined question asks for the maximum number of edges in a
$K_{p+1}$-free graph that is not $p$-partite.
The case $p=2$, namely the corresponding result for non-bipartite triangle-free graphs, was proved
independently by Erd\H{o}s and Gallai, and by Andr\'asfai~\cite{Erdos19626}.
In 1981, Brouwer~\cite{Brouwer1981} refined Tur\'{a}n's theorem in the following form.

\begin{thm}[Brouwer~\cite{Brouwer1981}]\label{Brouwer}
For any integer $n \ge 2p+1$, let $G$ be an $n$-vertex graph that is not $p$-partite. If $G$ contains no copy of $K_{p+1}$, then
$e(G) \le e(T_{n,p}) - \lfloor \frac{n}{p}\rfloor + 1.$
\end{thm}

Theorem~\ref{Brouwer} has also been independently investigated in various contexts; see, for instance, \cite{Amin2013, Kang2005, Tyomkyn2015}.
The bound in Theorem~\ref{Brouwer} is sharp. Its sharpness follows from the
split construction in \cite[Definition~1.3]{LP2023}: one adds an edge inside
one part of a Tur\'an graph and splits another part into two nonempty sets,
deleting the edges from the two endpoints of the added edge to these two sets,
respectively. The resulting graph is non-$p$-partite and $K_{p+1}$-free, with
exactly $e(T_{n,p})-\lfloor n/p\rfloor+1$ edges.

Thus, under the non-$p$-partite assumption, every graph with more than $e(T_{n,p})-\lfloor n/p\rfloor+1$ edges must contain a copy of $K_{p+1}$.
As a quantitative extension, we aim to establish the corresponding exact counting result above this threshold.
Specifically, for a fixed positive integer $s$ and a fixed integer $p\ge 2$,
we determine a sharp lower bound for $\#K_{p+1}(G)$ among all non-$p$-partite graphs $G$ with $e(G)\ge e(T_{n,p})-\lfloor n/p\rfloor+s+1$.
To show the sharpness of this bound, we first describe an extremal construction before formally stating the
result.

\begin{definition}\label{DEF1.1}
Let $p, n$ and $s$ be integers such that $n \ge 2p \ge 4$ and $1\leq s \le \lfloor n/p \rfloor$.
Let $T_1, \dots, T_p$ denote the parts of the Tur\'an graph $T_{n,p}$ ordered such that $|T_1| \ge |T_2| \ge \dots \ge |T_p|$.
Choose two distinct vertices $u^*,v_1\in T_1$. Let $X_1\subseteq T_p$ with $|X_1|=s$,
and set $X_2=T_p\setminus X_1$.
We define $Y_{n,p,s}$ to be the graph obtained from $T_{n,p}$ by adding the edge $u^*v_1$ and then deleting all edges between $u^*$ and $X_2$; see Figure \ref{fig-1.1} for the
case $p=3$.
\end{definition}

\begin{figure}[!htbp]
\centering
\begin{tikzpicture}[scale=0.75, x=1.00mm, y=1.00mm, inner xsep=0pt, inner ysep=0pt, outer xsep=0pt, outer ysep=0pt]
\definecolor{L}{rgb}{0,0,0}
\definecolor{F}{rgb}{0,0,0}

\node[draw,dashed, line width=0.3mm, minimum width=45mm, minimum height=8mm,
label={[label distance=1mm]90:$T_1$}
] (rect1) at (10,0) {};
% 缁樺埗涓€涓暱鏂瑰舰
%\draw(8,8) node[anchor=base west]{\fontsize{12.23}{17.07}\selectfont $V_1$};

\node[draw, line width=0.3mm, minimum width=20mm, minimum height=5mm] (rect11) at (-5,0) {};
% 缁樺埗涓€涓暱鏂瑰舰

\draw(-16.3,-1.5) node[anchor=base west]{\fontsize{10.23}{17.07}\selectfont $T_1\setminus\{u^*,v_1\}$};

\node[circle,fill=F,draw,inner sep=0pt,minimum size=2mm] (u1) at (20.00,00.00) {};
\draw(19,2) node[anchor=base west]{\fontsize{10.23}{17.07}\selectfont $u^*$};

\node[circle,fill=F,draw,inner sep=0pt,minimum size=2mm] (u2) at (30.00,00.00) {};
\draw(29,2) node[anchor=base west]{\fontsize{10.23}{17.07}\selectfont $v_1$};

\node[draw,dashed, line width=0.3mm, minimum width=45mm, minimum height=9mm] (rect2) at (10,-30) {};
% 缁樺埗涓€涓暱鏂瑰舰
\draw(8,-42) node[anchor=base west]{\fontsize{12.23}{17.07}\selectfont $T_3$};

\node[draw, line width=0.3mm, minimum width=20mm, minimum height=5mm] (X2) at (-5,-30) {$X_2$};
% 缁樺埗涓€涓暱鏂瑰舰

\node[draw, line width=0.3mm, minimum width=20mm, minimum height=5mm] (X1) at (25,-30) {$X_1$};
% 缁樺埗涓€涓暱鏂瑰舰

%\node[draw, line width=0.3mm, minimum width=10mm, minimum height=5mm] (X3) at (30,-30) {$X_3$}; % 缁樺埗涓€涓暱鏂瑰舰

\node[draw, line width=0.3mm, minimum width=20mm, minimum height=7mm,
] (T2) at (47,-15) {$T_2$}; % 缁樺埗涓€涓暱鏂瑰舰

\definecolor{L}{rgb}{0,0,0}
\path[line width=0.45mm, draw=blue] (u1) -- (u2);
\path[line width=0.45mm, draw=red] (u1) -- (X1);
\path[line width=0.45mm, draw=L] (u2) -- (X2);
\path[line width=0.45mm, draw=L] (u2) -- (X1);
\path[line width=0.45mm, draw=L] (rect11) -- (X1);
\path[line width=0.45mm, draw=L] (rect11) -- (X2);

\path[line width=0.45mm, draw=L] (u1) -- (T2);
\path[line width=0.45mm, draw=L] (u2) -- (T2);
\path[line width=0.45mm, draw=L] (rect11) -- (T2);
\path[line width=0.45mm, draw=L] (X1) -- (T2);
\path[line width=0.45mm, draw=L] (X2) -- (T2);

\end{tikzpicture}
\caption{The graph $Y_{n,p,s}$ for $p=3$.}{\label{fig-1.1}}
\end{figure}
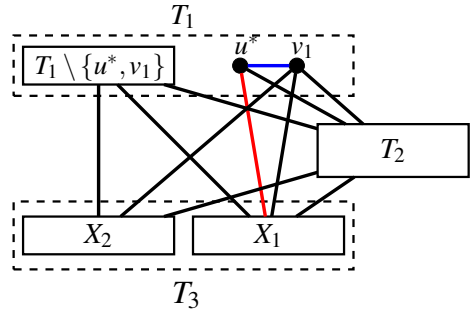

For any $s\in\big\{1,\ldots,\big\lfloor\frac {n}{p}\big\rfloor\big\}$,
we have
\begin{equation}\label{ALIGN-1}
e(Y_{n,p,s})=e(T_{n,p})-\Big\lfloor\frac{n}{p}\Big\rfloor+s+1.
\end{equation}
It is clear that the number of copies of $K_{p+1}$ in $Y_{n,p,s}$ is
\begin{align}\label{ALIGN-2}
\#K_{p+1}(Y_{n,p,s})=
\begin{cases}
 \begin{array}{ll}
s  & \hbox{if $p=2$,} \\
s\prod_{i=2}^{p-1}|T_i|=s\big(\frac{n}{p}\big)^{p-2}+O(n^{p-3})  & \hbox{if $p\geq 3$.}
\end{array}
\end{cases}
\end{align}
Since $s\ge 1$, the graph $Y_{n,p,s}$ contains a copy of $K_{p+1}$ and hence is non-$p$-partite.
Thus $Y_{n,p,s}$ gives a natural extremal construction for the clique-counting problem above Brouwer's threshold.
We now state our main theorem.

\begin{thm}\label{thm1.1}
Let $p\geq 2$ and $s\geq 1$ be fixed integers, and let $n$ be sufficiently large.
If $G$ is an $n$-vertex non-$p$-partite graph with
$e(G)\geq e(Y_{n,p,s}),$
then $$\#K_{p+1}(G)\geq \#K_{p+1}(Y_{n,p,s}).$$
Moreover, the bound is sharp and is attained by the construction $Y_{n,p,s}$.
\end{thm}

Observe that Theorem~\ref{thm-Mubayi} addresses the edge-excess regime where
$e(G)\ge e(T_{n,p})+q$ for $q>0$.
In contrast, our theorem focuses on the regime defined by the threshold
$e(Y_{n,p,s})$.
For fixed $s$, this threshold lies below the Tur\'an number by a linear term.
Within this regime, the chromatic assumption is essential, since
$p$-partite graphs contain no copy of $K_{p+1}$.
Thus, Theorem~\ref{thm1.1} complements the classical edge-excess
supersaturation results.

On the other hand, Theorem~\ref{Brouwer} establishes an existence threshold.
For sufficiently large $n$, this threshold is
$h(n,p)=e(T_{n,p})-\lfloor n/p\rfloor+1$.
Hence Theorem~\ref{Brouwer} implies that every non-$p$-partite graph with more
than $h(n,p)$ edges contains a copy of $K_{p+1}$.
Our Theorem~\ref{thm1.1} provides the corresponding sharp counting version:
since $e(Y_{n,p,s})=h(n,p)+s$, it yields a sharp lower bound on the number
of copies of $K_{p+1}$ in every non-$p$-partite graph with at least
$h(n,p)+s$ edges.
In particular, when $s=1$, this result strengthens the existence conclusion of
Theorem~\ref{Brouwer} to an exact clique-counting statement.

\section{Preliminaries}
In this section, we present some preliminary results that will be used in the
proof.

\begin{thm}\label{thm}\emph{(\cite{Erdos1964})}
Let $\ell\ge 0$. If an $n$-vertex graph $G$ with
$e(G)=\big\lfloor \frac{n^2}{4}\big\rfloor-\ell$ contains a triangle, then $G$
contains at least $\big\lfloor \frac{n}{2}\big\rfloor-\ell-1$ triangles.
\end{thm}
Theorem~\ref{thm} is a classical exact counting result below the Mantel
threshold  $e(T_{n,2})=\big\lfloor \frac{n^2}{4}\big\rfloor$, the maximum
number of edges in an $n$-vertex triangle-free graph. It has inspired several later developments: Moon~\cite{Moon1965}
extended this type of counting problem from triangles to complete subgraphs,
while more recent work of Li, Lu and Peng~\cite{LLP}, and of Li, Feng and
Peng~\cite{LFP}, developed spectral analogues and spectral supersaturation
results for triangles and related subgraphs.

We next recall two standard structural tools for graphs near the Tur\'an
threshold. The first is the Graph Removal Lemma, which turns a small number of
forbidden subgraphs into a nearby graph with no such subgraph.

\begin{lem}\label{LEM2.1}\emph{(Graph Removal Lemma \cite{Komlos1996})}
Suppose that an $n$-vertex graph $G$ has at most $o(n^{|F|})$ copies of $F$.
Then there exists a set of edges in $G$ of size $o(n^2)$ whose removal from
$G$ results in an $F$-free graph.
\end{lem}

After applying Lemma \ref{LEM2.1}, one obtains an $F$-free graph after
deleting only few edges. When this graph still has nearly Tur\'an many edges,
the following classical stability theorem of Erd\H{o}s \cite{Erdos-1967,Erdos-1968} and Simonovits \cite{Simonovits1966} shows
that it must be close to the corresponding Tur\'an graph.

\begin{lem} \label{lem2.2}
\emph{(\cite{Erdos-1967,Erdos-1968,Simonovits1966})}
Let $H$ be a given graph with $\chi(H)=p+1$, and let $G$ be an $H$-free graph
on $n$ vertices. For any $\varepsilon>0$ and integer $p\geq 2$, there exist
a constant $\delta>0$ and an integer $n_0$ such that if $n\geq n_0$ and
$e(G)\geq (\frac{p-1}{p}-\delta)\frac{n^2}{2}$, then $G$ differs
from $T_{n,p}$ by at most $\varepsilon n^2$ edges.
\end{lem}

\section{Proof of Theorem \ref{thm1.1}}\label{section4}

In this section, we present the proof of Theorem \ref{thm1.1}.

\begin{proof}[\textbf{Proof of Theorem \ref{thm1.1}}]
Let $G$ be an $n$-vertex non-$p$-partite graph with
$e(G)\geq e(Y_{n,p,s}).$
If $G\cong Y_{n,p,s}$, then $e(G)=e(Y_{n,p,s})$ and $\#K_{p+1}(G)=\#K_{p+1}(Y_{n,p,s})$.
Thus, it is enough to prove that $\#K_{p+1}(G)\geq \#K_{p+1}(Y_{n,p,s}).$

We first consider the case $p=2$.
Let $G$ be an $n$-vertex non-bipartite graph satisfying
$e(G)\ge \lfloor \frac{(n-1)^2}{4}\rfloor+s+1.$
Because $G$ is non-bipartite, it contains an odd cycle.
Fix an odd cycle $C$ in $G$.
For sufficiently large \(n\),
$\lfloor \frac{(n-1)^2}{4}\rfloor+s+1\ge n\ge |E(C)|$.
Therefore, we can choose a spanning subgraph $H$ of $G$ such that
$C\subseteq H$ and
$e(H)=\lfloor \frac{(n-1)^2}{4}\rfloor+s+1$.
In particular, $H$ remains non-bipartite.
Since $s\ge 1$, we obtain
$e(H)=\lfloor \frac{(n-1)^2}{4}\rfloor+s+1
\geq e(T_{n,2})-\lfloor \frac{n}{2}\rfloor +2$.
Then by Theorem \ref{Brouwer}, \(H\) contains a triangle.
Let $\ell=e(T_{n,2})-e(H).$
Then
$
\ell
=e(T_{n,2})
-\Big\lfloor \frac{(n-1)^2}{4}\Big\rfloor-s-1
=
\Big\lfloor \frac n2\Big\rfloor-s-1>0.
$
Since $H$ contains a triangle and $\ell>0$, Theorem~\ref{thm} yields
$\#K_3(H)\ge
\lfloor \frac n2\rfloor-\ell-1
=s$.
Furthermore, using \eqref{ALIGN-2} we have
$$\#K_3(G)\ge \#K_3(H)\ge s=\#K_3(Y_{n,2,s}),$$
as desired.

It remains to consider the case $p\ge 3$.
Let $\varepsilon>0$ be a sufficiently small constant and $n$ be sufficiently large.
Among all $n$-vertex non-$p$-partite graphs $G$ satisfying
$e(G)\geq e(Y_{n,p,s})$,
we choose one that minimizes the number of copies of $K_{p+1}$. Subject to this choice, we further choose $G$ so that $e(G)$ is as large as possible.
In the following, we present the proof via a sequence of claims.

\begin{claim}\label{CLAIM3.1}
$G$ can be obtained from $T_{n,p}$ by adding and deleting at most $2\varepsilon n^2$ edges.
\end{claim}

\begin{proof}
By the choice of $G$ and \eqref{ALIGN-2}, we have
\begin{align}\label{ALIGN-3}
\#K_{p+1}(G) \leq \#K_{p+1}(Y_{n,p,s})\leq s\Big(\frac np\Big)^{p-2}+O(n^{p-3}).
\end{align}
Using Lemma~\ref{LEM2.1}, there exists a set $E_0\subseteq E(G)$ with
$|E_0|=o(n^2)$ such that $G'=G-E_0$ is $K_{p+1}$-free.
For the given $\varepsilon > 0$, let $\delta_0 >0$ be the constant guaranteed by Lemma~\ref{lem2.2}.
By \eqref{ALIGN-1}, we have
\begin{align}\label{EQU-2B}
e(G)\geq e(Y_{n,p,s})=e(T_{n,p})-\big\lfloor\frac np\big\rfloor+s+1
=\frac{p-1}{2p}n^2-O(n).
\end{align}
Furthermore, since $|E_0|=o(n^2)$, it follows that
$e(G')=e(G)-|E_0|\geq( \frac{p-1}{p}-\delta_0) \frac{n^2}{2}.$
Then by Lemma~\ref{lem2.2},
$G'$ differs from $T_{n,p}$ by at most $\varepsilon n^2$ edges.
Consequently, since $G$ is obtained from $G'$ by adding back the edges in $E_0$,
and noting that $|E_0| \le \varepsilon n^2$ for sufficiently large $n$,
the graph $G$ differs from $T_{n,p}$ by at most
$|E_0|+\varepsilon n^2 \le 2\varepsilon n^2$
edges.
This completes the proof of Claim~\ref{CLAIM3.1}.
\end{proof}

For any subset $S \subseteq V(G)$, we let $E(S)$ denote the edge set of the subgraph induced by $S$. If two subsets $S, T \subseteq V(G)$ are disjoint, $E(S,T)$ denotes the set
of edges with one endpoint in $S$ and the other in $T$. For simplicity, we write $e(S) = |E(S)|$ and $e(S,T) = |E(S,T)|$. The neighborhood and degree of a vertex $v \in V(G)$
are denoted by $N_G(v)$ and $d_G(v)$, respectively. For $S \subseteq V(G)$, we define $N_S(v) = N_G(v) \cap S$ and $d_S(v)=|N_S(v)|$.
Throughout this paper, we fix a constant $\zeta$ satisfying  $2\varepsilon<\zeta^3$ and $0<\zeta<\frac{1}{20p^4}$.

\begin{claim}\label{CLAIM3.2}
Choose a partition $V(G)=\bigcup_{i=1}^{p}U_i$ such that
$\sum_{1\leq i<j\leq p} e(U_i,U_j)$ is maximized over all $p$-partitions of $V(G)$.
Then $\sum_{i=1}^{p}e(U_i)\leq \zeta^3 n^2$ and $\big||U_i|-\frac{n}{p}\big|\leq \zeta n$ for every $i\in[p]$,
where $[p]=\{1,2,\ldots,p\}$.
\end{claim}

\begin{proof}
From Claim \ref{CLAIM3.1} we know that $G$ can be obtained from $T_{n,p}$ by adding and deleting at most $2\varepsilon n^2$ edges.
Let
$V(G)=\bigcup_{i=1}^{p}V_i$
be the natural vertex partition of $T_{n,p}$.
Then
$\big\lfloor \frac{n}{p}\big\rfloor
\leq |V_i|\leq
\big\lceil \frac{n}{p}\big\rceil$
for every $i\in[p]$,
and $T_{n,p}$ has no edges inside any $V_i$.
Hence all edges of $G[V_i]$ come from the added edges, and therefore
$\sum_{i=1}^{p}e(V_i)\leq 2\varepsilon n^2<\zeta^3 n^2.$
Now choose a vertex partition
$V(G)=\bigcup_{i=1}^{p}U_i$,
which minimizes $\sum_{i=1}^{p}e(U_i)$ over all $p$-partitions of $V(G)$.
Equivalently, this partition maximizes $\sum_{1\leq i<j\leq p}e(U_i,U_j)$.
By the choice of $U_1,\dots,U_p$, we have
$$
\sum_{i=1}^{p}e(U_i)\leq\sum_{i=1}^{p}e(V_i)\leq \zeta^3 n^2.
$$

Set $\beta=\max_{i\in[p]}\big||U_i|-\frac{n}{p}\big|.$
Since $\sum_{i=1}^{p}|U_i|=n$, we have
\begin{align*}
\beta^2
\leq
\sum_{i=1}^{p}\Big(|U_i|-\frac{n}{p}\Big)^2
=
\sum_{i=1}^{p}|U_i|^2
-\frac{2n}{p}\sum_{i=1}^{p}|U_i|
+p\cdot \frac{n^2}{p^2}
=\sum_{i=1}^{p}|U_i|^2-\frac{n^2}{p}.
\end{align*}
Thus
$\sum_{i=1}^{p}|U_i|^2\geq \frac{n^2}{p}+\beta^2.$
Consequently,
\begin{align*}
\sum_{1\leq i<j\leq p}|U_i||U_j|
&=
\frac{1}{2}\left(
\Big(\sum_{i=1}^{p}|U_i|\Big)^2
-\sum_{i=1}^{p}|U_i|^2
\right)
\leq
\frac{p-1}{2p}n^2-\frac{\beta^2}{2}.
\end{align*}
Using the bound on the number of edges inside the parts, we obtain
\begin{align*}
e(G)
=
\sum_{1\leq i<j\leq p}e(U_i,U_j)
+\sum_{i=1}^{p}e(U_i)
\leq
\sum_{1\leq i<j\leq p}|U_i||U_j|
+\zeta^3 n^2
\leq
\frac{p-1}{2p}n^2-\frac{\beta^2}{2}
+\zeta^3 n^2.
\end{align*}

On the other hand, using \eqref{EQU-2B} and the assumption that $n$ is sufficiently large,
we have
$e(G)\geq e(T_{n,p})-\zeta^3 n^2
>(\frac{p-1}{2p}-2\zeta^3)n^2.$
Combining the last two inequalities gives
$(\frac{p-1}{2p}-2\zeta^3)n^2
<\frac{p-1}{2p}n^2-\frac{\beta^2}{2}+\zeta^3 n^2.$
Hence
$\beta^2<6\zeta^3 n^2.$
Since $\zeta<\frac{1}{20p^4}<\frac16$, we have
$\beta<\zeta n.$
By the definition of $\beta$, this means that
$\big||U_i|-\frac{n}{p}\big|<\zeta n$
 for every $i\in[p]$.
This completes the proof of Claim \ref{CLAIM3.2}.
\end{proof}

Following Claim~\ref{CLAIM3.2},
we fix the vertex partition $V(G)=\bigcup_{i=1}^{p}U_i$ of $G$ as defined therein.
Set $W=\{v\in V(G): d_G(v)\leq (\frac{p-1}{p}-10\zeta)n\}$.
For each $i\in[p]$, let
$R_i=\{v\in U_i: d_{U_i}(v)\geq 5\zeta n\},$
and set $R=\bigcup_{i=1}^{p}R_i$. We next establish upper bounds on $|W|$ and $|R|$ through the following two claims.

\begin{claim}\label{CLAIM3.3}
$|W|\leq \zeta n$.
\end{claim}

\begin{proof}

Suppose to the contrary that $|W|>\zeta n$.
Choose a subset $W_0\subseteq W$ with
$|W_0|=\lfloor \zeta n\rfloor,$
and let
$H=G-W_0.$
Set $n'=|H|=n-\lfloor \zeta n\rfloor$.
Since every vertex $v\in W_0$ satisfies
$d_G(v)\leq \big(\frac{p-1}{p}-10\zeta\big)n$
and noting that $|W_0|=\lfloor \zeta n\rfloor,$
it follows that
\begin{align*}
e(H)
\geq e(G)-\sum_{v\in W_0}d_G(v)
\geq e(G)-\lfloor \zeta n\rfloor
\Big(\frac{p-1}{p}-10\zeta\Big)n
\geq e(G)-\zeta n
\Big(\frac{p-1}{p}-10\zeta\Big)n.
\end{align*}
By \eqref{EQU-2B}, we have
$e(T_{n,p})-\big\lfloor\frac{n}{p}\big\rfloor+s+1
\geq(\frac{p-1}{2p}-\zeta^3)n^2.$
Therefore
\begin{align*}
e(H)
\geq\Big(\frac{p-1}{2p}-\zeta^3\Big)n^2-\zeta n\Big(\frac{p-1}{p}-10\zeta\Big)n
=\Big(
\frac{p-1}{2p}
-\frac{p-1}{p}\zeta
+10\zeta^2
-\zeta^3
\Big)n^2.
\end{align*}
Since $n'=n-\lfloor \zeta n\rfloor$, the choice of constants gives
$$
\left(
\frac{p-1}{2p}
-\frac{p-1}{p}\zeta
+10\zeta^2
-2\zeta^3
\right)n^2
>
\frac{p-1}{2p}(n')^2+\zeta^2 n^2.
$$
These, together with
$e(T_{n',p})\leq \frac{p-1}{2p}(n')^2,$
yield that
$e(H)>e(T_{n',p})+\zeta^2 n^2.$

Let $\delta=\delta_{K_{p+1}}>0$ be the constant given by Theorem \ref{thm-Mubayi}.
Since
$n'=n-\lfloor \zeta n\rfloor=\Theta(n)$, for sufficiently large $n$, define
$q_0=\lfloor \frac{\delta n'}{2}\rfloor .$
Then $1\leq q_0<\delta n'$. Moreover, since $q_0=O(n)$ while $\zeta^2 n^2=\Omega(n^2)$, we have
$e(H)>e(T_{n',p})+\zeta^2 n^2\geq e(T_{n',p})+q_0$
for sufficiently large $n$.
Applying Theorem \ref{thm-Mubayi} to $H$ with $F=K_{p+1}$ and $q=q_0$, we obtain
$\#K_{p+1}(H)\geq q_0\cdot c(n',K_{p+1}).$
Since $q_0=\Theta(n')$ and
$c(n',K_{p+1})=\Theta((n')^{p-1}),$
it follows that
$\#K_{p+1}(H)\geq \Omega((n')^p)=\Omega(n^p).$
As $H$ is an induced subgraph of $G$, we have
$\#K_{p+1}(G)\geq \#K_{p+1}(H)\geq \Omega(n^p)$,
which contradicts \eqref{ALIGN-3}. This contradiction proves that $|W|\leq \zeta n$.
\end{proof}

\begin{claim}\label{CLAIM3.4}
$|R|\leq \frac{1}{5}\zeta n$.
\end{claim}

\begin{proof}
For each $i\in[p]$, by the definition of $R_i$, every vertex $v\in R_i$ satisfies
$d_{U_i}(v)\geq 5\zeta n.$
Hence
$$
2e(U_i)=\sum_{v\in U_i}d_{U_i}(v)
\geq \sum_{v\in R_i}d_{U_i}(v)
\geq 5\zeta n |R_i|.
$$
Therefore,
$e(U_i)\geq \frac{5}{2}\zeta n |R_i|.$
Summing over all $i\in[p]$ and using Claim \ref{CLAIM3.2}, we obtain
$$
\zeta^3 n^2
\geq \sum_{i=1}^{p}e(U_i)
\geq \frac{5}{2}\zeta n \sum_{i=1}^{p}|R_i|
=
\frac{5}{2}\zeta n |R|.
$$
It follows that
$|R|\leq \frac{2}{5}\zeta^2 n.$
Since $\zeta\leq \frac{1}{2}$, we further have
$|R|\leq \frac{1}{5}\zeta n.$
\end{proof}

Define $\overline{U}_i=U_i\setminus (W\cup R)$ for every $i\in [p]$.
Given a $p$-partition $V(G)=U_1\cup U_2\cup\cdots\cup U_p$, an edge $e\in E(G)$ is called a \emph{class-edge} if its endpoints lie in the same part $U_i$ for some $i\in[p]$.
Otherwise, $e$ is called a \emph{cross-edge}.
With this notation, we obtain the following result.

\begin{claim}\label{CLAIM3.5}
Let $i_0, j_0\in [p]$ be two integers. Then

\vspace{1mm}
\textnormal{(i)} $d_{U_{i_0}}(u)\geq\big(\frac{1}{p^2}-4p\zeta\big)n$ for any $u\in \bigcup_{k\in [p]\setminus \{i_0\}}(R_{k}\setminus W)$;

\vspace{1mm}
\textnormal{(ii)} $d_{U_{i_0}}(u)\geq\big(\frac{1}{p}-6p\zeta\big)n$ for any $u\in \bigcup_{k\in [p]\setminus \{i_0\}}\overline{U}_{k}$;

\vspace{1mm}
\textnormal{(iii)} if $u\in \bigcup_{k\in [p]\setminus \{i_0\}}(R_{k}\setminus W)$ and
$\{u_1,\dots,u_{j_0}\}\subseteq \bigcup_{k\in [p]\setminus \{i_0\}}\overline{U}_k$, then there are at least $\frac{n}{2p^2}$ vertices in $\overline{U}_{i_0}$ adjacent to
$u_1,\dots,u_{j_0}$ and $u$.
\end{claim}

\begin{proof}
(i) Assume that $u\in R_{k_0}\setminus W$ for some ${k_0}\in [p]\setminus \{i_0\}$.
Since $V(G)=\bigcup_{i=1}^{p}U_{i}$ is a partition such that $\sum_{1\leq i<j\leq p}e(U_{i},U_{j})$ is the maximum,
it follows that $d_{U_{k_0}}(u)\le d_{U_i}(u)$ for every $i\ne k_0$.
Therefore $d_{U_{{k_0}}}(u)\leq \frac{1}{p}d_{G}(u)$.
It follows from Claim \ref{CLAIM3.2} that $|U_{k}|\leq \big(\frac{1}{p}+\zeta\big)n$ for any $k\in [p]$.
These, together with $d_G(u)> \big(\frac{p-1}{p}-10\zeta\big)n$ (as $u\notin W$), give that
\begin{align*}
  d_{U_{i_0}}(u)
  &= d_G(u)-d_{U_{k_0}}(u)-\sum\limits_{k\in [p]\setminus\{i_0,k_0\}}d_{U_{k}}(u)\\
    &\geq \frac{p-1}{p}d_G(u)-(p-2)\Big(\frac{1}{p}+\zeta\Big)n
    \geq \Big(\frac{1}{p^2}-4p\zeta\Big)n.
\end{align*}

(ii) Assume that $u\in \overline{U}_{k_0}$ for some $k_0\in [p]\setminus \{i_0\}$.
 Since $u\notin W$, we get $d_G(u)> \big(\frac{p-1}{p}-10\zeta\big)n$;
Since $u\notin R_{k_0}$, we have $d_{U_{k_0}}(u)<5\zeta n$.
Then it follows from $d_{U_k}(u)\leq |U_{k}|\leq \big(\frac{1}{p}+\zeta\big)n$ for any $k\in [p]\setminus\{i_0,k_0\}$ that
\begin{align*}
  d_{U_{i_0}}(u)
    &= d_G(u)-d_{U_{k_0}}(u)-\sum\limits_{k\in [p]\setminus\{i_0,k_0\}}d_{U_{k}}(u)\\
    &\geq \Big(\frac{p-1}{p}-10\zeta\Big)n-5\zeta n- (p-2)\Big(\frac{1}{p}+\zeta\Big)n
    \geq \Big(\frac{1}{p}-6p\zeta\Big)n.
\end{align*}

(iii) In view of (i), we have
$d_{U_{i_0}}(u)\geq (\frac{1}{p^2}-4p\zeta)n.$
Moreover, since each $u_k$ belongs to $\bigcup_{\ell\in[p]\setminus\{i_0\}}\overline{U}_{\ell}$,
it follows from (ii) that
$d_{U_{i_0}}(u_k)\geq (\frac{1}{p}-6p\zeta)n$
for every  $k\in[j_0]$.
Let
$X=N_{U_{i_0}}(u)\cap \bigcap_{k=1}^{j_0}N_{U_{i_0}}(u_k).$
Indeed, we use the following elementary fact.
\begin{align*}
\Big|\bigcap_{k=1}^{j_0}N_{U_{i_0}}(u_k)\Big|
&=|U_{i_0}|-\Big|U_{i_0}\setminus\bigcap_{k=1}^{j_0}N_{U_{i_0}}(u_k)\Big|\\
&\ge|U_{i_0}|-\sum_{k=1}^{j_0}\Big(|U_{i_0}|-|N_{U_{i_0}}(u_k)|\Big)\\
&=\sum_{k=1}^{j_0}|N_{U_{i_0}}(u_k)|-(j_0-1)|U_{i_0}|.
\end{align*}
Consequently,
\begin{align*}
|X|
&\geq |N_{U_{i_0}}(u)|
+\sum_{k=1}^{j_0}|N_{U_{i_0}}(u_k)|
-j_0|U_{i_0}| \\
&\geq
\Big(\frac{1}{p^2}-4p\zeta\Big)n
+j_0\Big(\frac{1}{p}-6p\zeta\Big)n
-j_0\Big(\frac{1}{p}+\zeta\Big)n \\
&=
\Big(\frac{1}{p^2}-4p\zeta-j_0(6p+1)\zeta\Big)n.
\end{align*}
Since $j_0\leq p$, and using $p\geq 3$, we have
$$
4p+j_0(6p+1)\leq 4p+p(6p+1)=6p^2+5p\leq 8p^2.
$$
Thus
$|X|\geq (\frac{1}{p^2}-8p^2\zeta)n.$
By Claims \ref{CLAIM3.3} and \ref{CLAIM3.4},
$|W\cup R|\leq \zeta n+\frac{1}{5}\zeta n=\frac{6}{5}\zeta n$.
Combined with the assumption $\zeta<\frac{1}{20p^4}$, we deduce
$\Big(\frac{1}{p^2}-8p^2\zeta\Big)n
\geq\frac{6}{5}\zeta n+\frac{n}{2p^2}
\geq |W\cup R|+\frac{n}{2p^2}
$.
We thus conclude that
$|X|\geq |W\cup R|+\frac{n}{2p^2}$.
Since $\overline{U}_{i_0}=U_{i_0}\setminus (W\cup R)$, it follows that
$|X\cap \overline{U}_{i_0}|
\geq |X|-|W\cup R|
\geq \frac{n}{2p^2}.$
Hence there are at least $\frac{n}{2p^2}$ vertices in $\overline{U}_{i_0}$ adjacent to all of $u,u_1,\ldots,u_{j_0}$.
\end{proof}

\begin{claim}\label{CLAIM3.6}
Let $i\in[p]$,
$u\in \overline{U}_i\cup (R_i\setminus W)$ and $v\in \overline{U}_i$.
If $uv \in E(G[U_i])$, then
$\#K_{p+1}(G,uv)\geq (\frac{n}{3p^2})^{p-1}$,
where $\#K_{p+1}(G,uv)$ denotes the number of copies of $K_{p+1}$ in $G$
that contain the edge $uv$.
\end{claim}

\begin{proof}
Suppose that $uv \in E(G[U_i])$. Without loss of generality, we may assume $i=1$.
We will construct copies of $K_{p+1}$ containing the edge $uv$ by sequentially selecting one vertex from each of $\overline U_2, \ldots, \overline U_p$.

Specifically, we recursively define singletons $\widehat U_k=\{x_k\} \subseteq \overline U_k$ for $2 \leq k \leq p$ such that $x_k$ is adjacent to every vertex in
$\bigcup_{t=1}^{k-1} \widehat U_t$. Set $\widehat U_1=\{u, v\}$. We shall show that, at each step $k$, there are at least $\lfloor \frac{n}{2p^2} \rfloor$ candidates for $x_k$.
We proceed by analyzing two cases based on the location of $u$.

\medskip
\noindent\textbf{Case 1.} $u \in R_1 \setminus W$.

Fix $k\in\{2,\ldots,p\}$ and assume
that $x_2,\ldots,x_{k-1}$ have already been chosen. Then we have
$$\{v,x_2,\ldots,x_{k-1}\}\subseteq \bigcup_{j\in[p]\setminus\{k\}}\overline U_j.$$
Furthermore, since $u \in R_1 \setminus W$ and $k \neq 1$, it follows that $u \in \bigcup_{j \in [p] \setminus \{k\}} (R_j \setminus W)$. Applying Claim~\ref{CLAIM3.5} (iii)
with $i_0 = k$, we obtain at least $\frac{n}{2p^2}$ vertices in $\overline U_k$ that are adjacent to all of $u, v, x_2, \ldots, x_{k-1}$. Thus, there are at least $\lfloor
\frac{n}{2p^2} \rfloor$ choices for $x_k$.

\medskip
\noindent\textbf{Case 2.} $u \in \overline U_1$.

Again, fix $k \in \{2, \ldots, p\}$ and assume that $x_2, \ldots, x_{k-1}$ have been chosen. Let $S = \bigcup_{t=1}^{k-1} \widehat U_t = \{u, v, x_2, \ldots, x_{k-1}\}$. Then
$|S| = k$, and every vertex in $S$ belongs to $\bigcup_{j \in [p] \setminus \{k\}} \overline U_j$. By Claim~\ref{CLAIM3.5} (ii), for each $x \in S$, we have $d_{U_k}(x) \geq
\left(\frac{1}{p} - 6p\zeta\right)n$.
Thus,
\begin{align*}
\Big| \bigcap_{x\in S}N_{U_k}(x) \Big|
&\ge k\Big(\frac1p-6p\zeta\Big)n-(k-1)\Big(\frac1p+\zeta\Big)n  \\
&= \Big(\frac1p-(6pk+k-1)\zeta\Big)n  \\
&\ge\Big(\frac1p-7p^2\zeta\Big)n.
\end{align*}
Deleting the vertices in $W\cup R$, and using $|W\cup R|\le \frac65\zeta n$ and $\zeta<\frac1{20p^4}$,
we obtain
 $$\Big|\bigcap_{x\in S}N_{\overline U_k}(x)\Big|\ge\Big(\frac1p-7p^2\zeta-\frac65\zeta\Big)n\ge\frac{n}{2p^2}.$$
Hence there are at least
$\big\lfloor \frac{n}{2p^2}\big\rfloor$ choices for $x_k$.

In either case, we can recursively choose $x_k \in \overline U_k$ for each $k = 2, \ldots, p$ with at least $\lfloor\frac{n}{2p^2}\rfloor$ choices at each step. Every resulting
set $\{u, v, x_2, \ldots, x_p\}$ spans a copy of $K_{p+1}$ containing the edge $uv$. Therefore,
for sufficiently large $n$,
$\#K_{p+1}(G,uv)
\geq
\lfloor \frac{n}{2p^2}\rfloor^{p-1}
\geq
\big(\frac{n}{3p^2}\big)^{p-1}$,
as desired.
\end{proof}

\begin{claim}\label{CLAIM3.7}
For every $i\in[p]$, we have
$E\big(G[\overline U_i]\big)=\varnothing$ and $R_i\subseteq W$.
\end{claim}

\begin{proof}
We first prove that $R_i\subseteq W$ for every $i\in[p]$.
Suppose to the contrary that there exists some $i_0\in[p]$ such that
$R_{i_0}\setminus W\neq \varnothing.$
Choose a vertex $u\in R_{i_0}\setminus W$. Since $u\in R_{i_0}$, we have
$d_{U_{i_0}}(u)\geq 5\zeta n.$
Applying Claims \ref{CLAIM3.3} and \ref{CLAIM3.4}, we obtain
$|W\cup R|\leq |W|+|R|\leq \frac{6}{5}\zeta n.$
Since $\overline U_{i_0}=U_{i_0}\setminus (W\cup R)$, it follows that
\begin{align*}
d_{\overline U_{i_0}}(u)
\geq d_{U_{i_0}}(u)-|W\cup R|
\geq 5\zeta n-\frac{6}{5}\zeta n
>0.
\end{align*}
Hence there exists a vertex $v\in \overline U_{i_0}$ such that
$uv\in E\big(G[U_{i_0}]\big).$
By Claim \ref{CLAIM3.6}, we have
$\#K_{p+1}(G)\geq \#K_{p+1}(G,uv)\geq \big(\frac{n}{3p^2}\big)^{p-1},$
which contradicts \eqref{ALIGN-3}. Thus $R_i\subseteq W$ for every $i\in[p]$.

It remains to prove that $E\big(G[\overline U_i]\big)=\varnothing$ for every $i\in[p]$.
Suppose otherwise that there exists an edge
$uv\in E\big(G[\overline U_i]\big)$
for some $i\in[p]$. Since
$u\in \overline U_i\subseteq \overline U_i\cup (R_i\setminus W)$
and
$v\in \overline U_i$,
Claim \ref{CLAIM3.6} implies that
$\#K_{p+1}(G,uv)=\Omega(n^{p-1}).$
Consequently,
$\#K_{p+1}(G)\geq \#K_{p+1}(G,uv)=\Omega(n^{p-1}),$
again contradicting \eqref{ALIGN-3}. Hence
$E\big(G[\overline U_i]\big)=\varnothing$
for every $i\in[p]$. The proof is complete.
\end{proof}

\begin{claim}\label{CLAIM3.8}
There exists a vertex $u^*\in V(G)$ such that
$\chi(G-\{u^*\})=p.$
\end{claim}

\begin{proof}

We first prove that $W\neq\varnothing$. Suppose, to the contrary, that
$W=\varnothing$. By Claim \ref{CLAIM3.7}, we have $R_i\subseteq W$ for every
$i\in[p]$, and hence
$
R=\bigcup_{i=1}^p R_i=\varnothing.
$
It follows that $\overline U_i=U_i$ for every $i\in[p]$.
Again by Claim~\ref{CLAIM3.7}, $E(G[\overline U_i])=\varnothing$, and
therefore $E(G[U_i])=\varnothing$ for every $i\in[p]$.
Thus $U_1,\ldots,U_p$ form a proper $p$-partition of $G$, contradicting the
assumption that $G$ is non-$p$-partite. Therefore, $W\neq\varnothing$.

Since $e(G)\geq e(Y_{n,p,s})$,
 Theorem \ref{Brouwer} implies that $G$ contains a
copy of $K_{p+1}$. Fix such a copy $F$. Since
$V(G)=U_1\cup\cdots\cup U_p$, the pigeonhole principle implies that $F$
contains two vertices belonging to the same part $U_i$ for some $i\in[p]$.
Since $F$ is a clique, these two vertices form an edge in $G[U_i]$.

Moreover, Claim \ref{CLAIM3.7} gives
$R=\bigcup_{j=1}^p R_j\subseteq W$
and $E(G[\overline U_i])=\varnothing$. Hence every vertex in
$U_i\setminus W$ lies in $\overline U_i$. Consequently, every edge of
$G[U_i]$ has at least one endpoint in $W$. In particular,
$V(F)\cap W\neq\varnothing$.
Choose $u^*\in V(F)\cap W$, and let $i_0\in[p]$ be such that
$u^*\in U_{i_0}$. Since $F-\{u^*\}\cong K_p$, the graph $G-\{u^*\}$ contains a copy
of $K_p$. Therefore,
$\chi\big(G-\{u^*\}\big)\geq p.$
If $\chi\big(G-\{u^*\}\big)=p$, then we are done.
Suppose, for a contradiction, that
$\chi\big(G-\{u^*\}\big)\geq p+1.$

Construct a graph $G'$ from $G$ by deleting all edges incident with $u^*$ and
then joining $u^*$ to every vertex in
$
\cup_{i\in[p]\setminus\{i_0\}}\overline U_i.
$
Since no edge not incident with $u^*$ is changed, we have
$G'-\{u^*\}=G-\{u^*\}.$
It follows that
$$
\chi(G')
\geq
\chi\big(G'-\{u^*\}\big)
=
\chi\big(G-\{u^*\}\big)
\geq p+1.
$$
Thus $G'$ is non-$p$-partite.

We next compare the numbers of edges in $G$ and $G'$. By Claims
\ref{CLAIM3.2}, \ref{CLAIM3.3}, and \ref{CLAIM3.4}, we have
$$
\begin{aligned}
\sum_{i\neq i_0}|\overline U_i|
\geq n-|U_{i_0}|-|W\cup R|
\geq n-\Big(\frac{1}{p}+\zeta\Big)n-\frac{6}{5}\zeta n
=\Big(\frac{p-1}{p}-\frac{11}{5}\zeta\Big)n.
\end{aligned}
$$
On the other hand, since $u^*\in W$, the definition of $W$ gives
$
d_G(u^*)
\leq
(\frac{p-1}{p}-10\zeta)n.
$
Since the only modified edges are those incident with $u^*$, we obtain
$$
\begin{aligned}
e(G')-e(G)
=
\sum_{i\neq i_0}|\overline U_i|-d_G(u^*)
\geq
\Big(\frac{p-1}{p}-\frac{11}{5}\zeta\Big)n
-
\Big(\frac{p-1}{p}-10\zeta\Big)n
=
\frac{39}{5}\zeta n
>0.
\end{aligned}
$$
Therefore,
$$
e(G')>e(G)
\geq
e\big(T_{n,p}\big)-\Big\lfloor\frac{n}{p}\Big\rfloor+s+1.
$$
Thus $G'$ satisfies the same edge condition as $G$ and is also
non-$p$-partite.

It remains to compare the numbers of copies of $K_{p+1}$ in $G$ and $G'$.
By construction,
$
N_{G'}(u^*)
=\cup_{i\in[p]\setminus\{i_0\}}\overline U_i.
$
By Claim \ref{CLAIM3.7}, each $\overline U_i$ is independent. Moreover, the
construction of $G'$ does not change any edge whose endpoints are both
different from $u^*$. Hence
$G'[N_{G'}(u^*)]$
is $(p-1)$-partite and therefore contains no copy of $K_p$. Consequently, no
copy of $K_{p+1}$ in $G'$ contains $u^*$.
Since $G'-\{u^*\}=G-\{u^*\}$, it follows that
$
\#K_{p+1}(G')
=\#K_{p+1}(G-\{u^*\}).$
On the other hand, $u^*$ belongs to the copy $F$ of $K_{p+1}$ in $G$, so
$
\#K_{p+1}(G-\{u^*\})
<\#K_{p+1}(G).
$
Therefore,
$
\#K_{p+1}(G')
<\#K_{p+1}(G),
$
contradicting the extremal choice of $G$. This contradiction shows that
$\chi\big(G-\{u^*\}\big)=p.$
The proof is complete.
\end{proof}

Among all pairs $(G,u)$ such that $G$ satisfies the previous two extremal
conditions and $\chi(G-\{u\})=p$, choose one for which $d_G(u)$ is minimum,
and denote it by $(G,u^*)$.
Now let $V_1,V_2,\ldots,V_p$ be the color classes of a proper $p$-coloring of
$G-\{u^*\}$.
Since $G-\{u^*\}$ is $p$-partite, it contains no copy of
$K_{p+1}$. Hence every copy of $K_{p+1}$ in $G$ must contain $u^*$. Since $G$
contains a copy of $K_{p+1}$,
 after relabeling the color classes if necessary,
there exist vertices $v_i\in V_i$ for each $i\in[p]$ such that
$G[\{u^*,v_1,\ldots,v_p\}]\cong K_{p+1}$.
Moreover, deleting $u^*$ gives a bijection between the copies of $K_{p+1}$ in
$G$ and the copies of $K_p$ in $G[N_G(u^*)]$.
Consequently,
$\#K_{p+1}(G)=\#K_p(G[N_G(u^*)])$.

Let $K$ be the complete $p$-partite graph with vertex classes
$V_1,V_2,\ldots,V_p$. Since each $V_i$ is an independent set in $G-\{u^*\}$, we
have $E(G-\{u^*\})\subseteq E(K)$. Define
$$
M=E(K)\setminus E(G-\{u^*\}).
$$
Thus $M$ is precisely the set of missing cross-edges of $G-\{u^*\}$ with respect
to the partition $V_1\cup V_2\cup\cdots\cup V_p$.

\begin{claim}\label{CLAIM3.9}
The graph $G-\{u^*\}$ is a complete $p$-partite graph.
\end{claim}

\begin{proof}
Suppose, to the contrary, that $G-\{u^*\}$ is not complete $p$-partite.
Then $M\neq\varnothing$, so there exists a missing cross-edge
$xy\in M$ with $x\in V_i$ and $y\in V_j$ for some distinct
$i,j\in[p]$. In particular,
$xy\notin E(G).$
Recall that $v_i\in V_i$, $v_j\in V_j$, and
$v_iv_j\in E(G)$. Hence $(x,y)\neq(v_i,v_j)$, and therefore
$x\neq v_i$ or $y\neq v_j$.
By interchanging the roles of $(x,V_i,v_i)$ and $(y,V_j,v_j)$
if necessary, we may assume that
$x\neq v_i.$
We distinguish two cases according to whether $x$ is adjacent to $u^*$.

\medskip
\noindent\textbf{Case 1.} $x\notin N_G(u^*)$.

Let $G'$ be the graph obtained from $G$ by adding the edge $xy.$
Then
$e(G')=e(G)+1.$
Moreover, the fixed copy of $K_{p+1}$ on
$\{u^*,v_1,\ldots,v_p\}$ remains unchanged, so $G'$ is non-$p$-partite.

We claim that adding $xy$ creates no new copy of $K_{p+1}$. Indeed, any
copy of $K_{p+1}$ in $G'$ that is not already present in $G$ must contain
the new edge $xy$. If such a copy avoids $u^*$, then it is contained in
$G'-\{u^*\}$. However, $G'-\{u^*\}$ is still $p$-partite with parts
$V_1,\ldots,V_p$, since $xy$ is a cross-edge between $V_i$ and $V_j$.
This is impossible.

On the other hand, if such a copy contains $u^*$, then, since it also
contains $x$, it must contain the edge $u^*x$. But
$x\notin N_G(u^*)$, and the edge $u^*x$ was not added in the construction
of $G'$. This is again impossible. Therefore,
$\#K_{p+1}(G')=\#K_{p+1}(G).$
Thus $G'$ is an $n$-vertex non-$p$-partite graph satisfying
$e(G')\geq e(Y_{n,p,s})$, with the same number of copies of $K_{p+1}$ as $G$ but
with more edges. This contradicts the choice of $G$ with $e(G)$ as large
as possible subject to minimizing $\#K_{p+1}(G)$.

\medskip
\noindent\textbf{Case 2.} $x\in N_G(u^*)$.

Let $G'$ be the graph obtained from $G$ by deleting the edge $u^*x$ and adding the edge $xy.$
Then $e(G')=e(G),$
so $G'$ still satisfies $e(G)\geq e(Y_{n,p,s})$. Since $x\neq v_i$, deleting the
edge $u^*x$ does not destroy the fixed copy of $K_{p+1}$ on
$\{u^*,v_1,\ldots,v_p\}$. Hence $G'$ remains non-$p$-partite.

We again claim that the addition of $xy$ creates no new copy of
$K_{p+1}$. Indeed, any copy of $K_{p+1}$ in $G'$ that is not present in
$G$ must contain $xy$. If it avoids $u^*$, then it is contained in
$G'-\{u^*\}$, which is $p$-partite with parts $V_1,\ldots,V_p$, a
contradiction. If it contains $u^*$, then, since it contains $x$, it
requires the edge $u^*x$, which was deleted in the construction of $G'$.
Thus no new copy of $K_{p+1}$ is created, and consequently
$\#K_{p+1}(G')\leq \#K_{p+1}(G).$

Since $G'$ is an $n$-vertex non-$p$-partite graph satisfying
$e(G')\geq e(Y_{n,p,s})$, the minimality of $\#K_{p+1}(G)$ implies that equality
$\#K_{p+1}(G')=\#K_{p+1}(G)$ must hold.
Together with $e(G')=e(G)$, this shows that $G'$ satisfies the same first
two extremal conditions as $G$.

Furthermore, $G'-\{u^*\}$ is $p$-partite with parts
$V_1,\ldots,V_p$. It also contains the copy of $K_p$ induced by
$\{v_1,\ldots,v_p\}$, since the only deleted edge is $u^*x$. Therefore,
$\chi\bigl(G'-\{u^*\}\bigr)=p.$
Hence $(G',u^*)$ is an admissible pair in the third extremal choice.
However,
$d_{G'}(u^*)=d_G(u^*)-1<d_G(u^*),$
contradicting the choice of $(G,u^*)$ with $d_G(u^*)$ as small as
possible.

In both cases we obtain a contradiction. Hence $M=\varnothing$. Therefore,
$G-\{u^*\}$ is a complete $p$-partite graph with parts
$V_1,\ldots,V_p$.
\end{proof}

Assume without loss of generality that $n_1\geq n_2\geq \cdots \geq n_p$, where $n_i=|V_i|$.
Set $k_i=|N_G(u^*)\cap V_i|$ for each $i\in[p]$.
Since $u^*$ is contained in a copy of $K_{p+1}$, and
$V_1,\ldots,V_p$ are the color classes of a proper $p$-coloring of
$G-\{u^*\}$, this copy must contain exactly one vertex from each $V_i$ besides $u^*$.
Hence $k_i\geq 1$ for every $i\in[p]$.

\begin{claim}\label{CLAIM3.10}
We have $n_1-n_p\leq 1$.
\end{claim}

\begin{proof}
Suppose to the contrary that $n_1-n_p\geq 2$.
Recall that $G-\{u^*\}$ is a complete $p$-partite graph with parts
$V_1,\ldots,V_p$, where $|V_i|=n_i$ and
$k_i=|N_G(u^*)\cap V_i|$ for every $i\in[p]$.

Since $G$ is non-$p$-partite, we have $k_i\geq 1$ for every $i\in[p]$.
Indeed, if $k_i=0$ for some $i$, then placing $u^*$ into $V_i$ would give a
$p$-partition of $G$, a contradiction. Moreover, since $G-\{u^*\}$ is
$p$-partite, every copy of $K_{p+1}$ in $G$ contains $u^*$. Therefore
$\#K_{p+1}(G)=k_1k_2\cdots k_p.$

Suppose first that $k_1\geq 2$. Choose a vertex
$w\in N_G(u^*)\cap V_1$. Let $G'$ be the graph obtained from $G$ by deleting
all edges between $w$ and $V_p\cup\{u^*\}$, and adding all edges between
$w$ and $V_1\setminus\{w\}$.
Let $V'_1=V_1\setminus\{w\}$,
$V'_p=V_p\cup\{w\}$, and
$V'_\ell=V_\ell$ for $\ell\in[p]\setminus\{1,p\}$.
By construction, $G'-\{u^*\}$ is a complete $p$-partite graph with parts
$V'_1,\ldots,V'_p$. Since $k_1\geq 2$, the vertex $u^*$ still has a neighbor
in $V'_1$ after the edge $u^*w$ is deleted. Also, $u^*$ has a neighbor in
each of the other parts. Hence $G'$ contains a copy of $K_{p+1}$, and so
$G'$ is non-$p$-partite.
Furthermore,
\begin{align*}
e(G')-e(G)
=|V_1\setminus\{w\}|-|V_p|-1
=(n_1-1)-n_p-1
=n_1-n_p-2\geq 0.
\end{align*}
Thus $G'$ satisfies the same edge condition as $G$.

The number of neighbors of $u^*$ in the new first part $V'_1$ is $k_1-1$,
while the numbers of neighbors of $u^*$ in all other parts are unchanged.
Since $G'-\{u^*\}$ is $p$-partite, every copy of $K_{p+1}$ in $G'$ contains
$u^*$. Hence
$
\#K_{p+1}(G')=(k_1-1)k_2\cdots k_p
<k_1k_2\cdots k_p=\#K_{p+1}(G),
$
contradicting the extremal choice of $G$.

It remains to consider the case $k_1=1$. Since $G-\{u^*\}$ is a complete
$p$-partite graph on $n-1$ vertices, we have
$e(G-\{u^*\})\le e(T_{n-1,p})$.
Since $K$ is a $p$-partite graph on $n-1$ vertices, we have
$e(K)\leq e(T_{n-1,p}).$
Therefore,
\begin{align*}
\sum_{i=1}^p k_i
=d_G(u^*)
&=e(G)-e(K)\\
&\geq e(T_{n,p})-\big\lfloor\frac np\big\rfloor+s+1-e(T_{n-1,p})\\
&=n-\big\lceil\frac np\big\rceil-\big\lfloor\frac np\big\rfloor+s+1.
\end{align*}
Since $p$ and $s$ are fixed and $n$ is sufficiently large, the last quantity is at least $p+1$.
Hence $\sum_{i=1}^p k_i\geq p+1$.
Since $k_1=1$ and $k_i\geq 1$ for every
$i\in[p]$, there exists an index $i\in[p]\setminus\{1\}$ such that $k_i\geq 2$.

Since $n_1-n_p\geq 2$ and $k_1=1$, we may choose a vertex
$w\in V_1\setminus N_G(u^*)$. Also choose a vertex
$z\in N_G(u^*)\cap V_i$. Let $G'$ be the graph obtained from $G$ by deleting
all edges between $w$ and $V_p$, deleting the edge $u^*z$, and adding all
edges between $w$ and $V_1\setminus\{w\}$.

Again let
$$
V'_1=V_1\setminus\{w\},\quad
V'_p=V_p\cup\{w\},\quad
V'_\ell=V_\ell
$$
for $\ell\in[p]\setminus\{1,p\}$.
Then $G'-\{u^*\}$ is a complete $p$-partite graph with parts
$V'_1,\ldots,V'_p$. Since $w\notin N_G(u^*)$, the number of neighbors of
$u^*$ in the new first part $V'_1$ is still $k_1=1$. Since $k_i\geq 2$, after
deleting the edge $u^*z$, the vertex $u^*$ still has a neighbor in the part
corresponding to $V_i$. The numbers of neighbors in all other parts remain
positive. Hence $G'$ contains a copy of $K_{p+1}$, and therefore $G'$ is
non-$p$-partite.
Moreover,
\begin{align*}
e(G')-e(G)
=|V_1\setminus\{w\}|-|V_p|-1
=(n_1-1)-n_p-1
=n_1-n_p-2\geq 0.
\end{align*}
Thus $G'$ also satisfies the same edge condition as $G$.

The construction decreases the number of neighbors of $u^*$ in $V_i$ by one
and leaves the corresponding neighbor counts in all other parts unchanged.
Consequently,
$$
\#K_{p+1}(G')=k_1\cdots k_{i-1}(k_i-1)k_{i+1}\cdots k_p
<k_1k_2\cdots k_p
=\#K_{p+1}(G),
$$
again contradicting the extremal choice of $G$.

Therefore our assumption $n_1-n_p\geq 2$ was false, and hence
$n_1-n_p\leq 1$.
\end{proof}

Recall that $T_1,\ldots,T_p$ are the partite sets of $T_{n,p}$ in
Definition \ref{DEF1.1}, ordered so that
$|T_1|\ge |T_2|\ge\cdots\ge |T_p|$, and that $u^*\in T_1$.
By Claims \ref{CLAIM3.9} and \ref{CLAIM3.10},
the graph $G-\{u^*\}$ is a complete $p$-partite graph whose part sizes
$n_1\geq n_2\geq\cdots\geq n_p$ differ by at most one. Therefore,
$G-\{u^*\}\cong T_{n-1,p}.$
By Definition \ref{DEF1.1},
it is clear that $|T_{i+1}|=n_i$ for $i\in [p-1]$ and $|T_1|=n_p+1$.
We reorder $k_1,\dots,k_p$ into $k_1',\dots,k_p'$,
where $k_1'\geq k_2'\geq\cdots\geq k_p'$.

\begin{claim}\label{CLAIM3.11}
We have $k_i'=n_i$ for $i\in [p-2]$, $k'_{p-1}=s$, and $k'_p=1$.
\end{claim}

\begin{proof}
We first prove that
\begin{align}\label{align-4}
e\big(Y_{n,p,s}\big)-e\big(T_{n-1,p}\big)
=\sum_{i=1}^{p-2}n_i+s+1.
\end{align}
In particular, the two smallest part sizes satisfy
$$
n_{p-1}+n_p
=|T_p|+(|T_1|-1)
=
\Big\lfloor\frac{n}{p}\Big\rfloor+
\Big\lceil\frac{n}{p}\Big\rceil
-1.
$$
Consequently,
$$
\begin{aligned}
\sum_{i=1}^{p-2}n_i
=(n-1)-(n_{p-1}+n_p)
=n-\Big\lceil\frac{n}{p}\Big\rceil
-\Big\lfloor\frac{n}{p}\Big\rfloor.
\end{aligned}
$$
Moreover, adding one vertex to a smallest part of $T_{n-1,p}$ produces
$T_{n,p}$, and hence
$
e(T_{n,p})-e(T_{n-1,p})
=n-\lceil\frac{n}{p}\rceil.
$
Combining these with \eqref{ALIGN-1},
we obtain
$$
\begin{aligned}
e\big(Y_{n,p,s}\big)-e\big(T_{n-1,p}\big)
&=
e\big(T_{n,p}\big)-e\big(T_{n-1,p}\big)
-
\Big\lfloor\frac{n}{p}\Big\rfloor
+s+1\\
&=
n-
\Big\lceil\frac{n}{p}\Big\rceil
-
\Big\lfloor\frac{n}{p}\Big\rfloor
+s+1
=
\sum_{i=1}^{p-2}n_i+s+1,
\end{aligned}
$$
which establishes \eqref{align-4}.

We then show that $e(G)=e\big(Y_{n,p,s}\big).$
Suppose otherwise. Since $G$ satisfies the edge condition, we have
$e(G)\ge e\big(Y_{n,p,s}\big)+1.$
Since $G-\{u^*\}\cong T_{n-1,p}$ and $e(G)=e\big(T_{n-1,p}\big)+\sum_{i=1}^{p}k_i,$
by \eqref{align-4}, we obtain
$$
\sum_{i=1}^{p}k_i>e\big(Y_{n,p,s}\big)-e\big(T_{n-1,p}\big)=\sum_{i=1}^{p-2}n_i+s+1.
$$
Since $n_i\ge1$ for every $i\in[p]$ and $s\ge1$, we have
$\sum_{i=1}^{p-2}n_i+s+1\ge p.$
Hence
$\sum_{i=1}^{p}k_i>p.$
Since $k_i\ge1$ for every $i\in[p]$, some $k_j$ is at least $2$.
Choose a vertex $z\in N_G(u^*)\cap V_j,$
and let $\widehat G$ be obtained from $G$ by deleting the edge $u^*z$.
Then $e(\widehat G)=e(G)-1\ge e\big(Y_{n,p,s}\big).$
Also, since $k_j\ge2$, the vertex $u^*$ still has a neighbor in every part of $G-\{u^*\}$.
Hence $\widehat G$ is still non-$p$-partite.
But
$$
\#K_{p+1}(\widehat G)=k_1\cdots k_{j-1}(k_j-1)k_{j+1}\cdots k_p<k_1k_2\cdots k_p=\#K_{p+1}(G),
$$
contradicting the extremal choice of $G$.
Therefore $e(G)=e\big(Y_{n,p,s}\big).$
It follows that
$$
\sum_{i=1}^{p}k_i=e(G)-e\big(T_{n-1,p}\big)=e\big(Y_{n,p,s}\big)-e\big(T_{n-1,p}\big)=\sum_{i=1}^{p-2}n_i+s+1.
$$

We now consider the rearranged vector $k_1'\ge k_2'\ge\cdots\ge k_p'$.
We first note that $1\le k_i'\le n_i$ for every $i\in[p]$.
Indeed, suppose that $k_i'>n_i$ for some $i\in[p]$.
Then at least $i$ of the numbers $k_1,\ldots,k_p$ are at least $k_i'$.
Hence at least $i$ of the numbers $n_1,\ldots,n_p$ are at least $k_i'$, since $k_j\le n_j$ for every $j\in[p]$.
 However, as $n_1\ge n_2\ge\cdots\ge n_p$ and $k_i'>n_i$,
there are at most $i-1$ such numbers among $n_1,\ldots,n_p$, a contradiction.
Thus $1\le k_i'\le n_i$ for every $i\in[p]$.
Moreover,
$$
\sum_{i=1}^{p}k_i'=\sum_{i=1}^{p}k_i=\sum_{i=1}^{p-2}n_i+s+1,
$$
and $\prod_{i=1}^{p}k_i'=\prod_{i=1}^{p}k_i.$
Since $G$ is chosen to minimize $\#K_{p+1}(G)$, the vector $(k_1',\ldots,k_p')$
minimizes $\prod_{i=1}^{p}x_i$ among all integer vectors $(x_1,\ldots,x_p)$ satisfying
$1\le x_i\le n_i$ for every $i\in[p],$ and $\sum_{i=1}^{p}x_i=\sum_{i=1}^{p-2}n_i+s+1.$
Indeed, if another such vector $(x_1,\ldots,x_p)$ had a smaller product,
then by starting with $T_{n-1,p}$ and adding a new vertex adjacent to exactly $x_i$ vertices in $V_i$ for each $i\in[p]$,
we would obtain an $n$-vertex non-$p$-partite graph satisfying the same edge condition and containing fewer copies of $K_{p+1}$, contradicting the choice of $G$.

We now prove that $k_i'=n_i$ for every $i\in[p-2]$.
Suppose not, and let $i'=\min\{i\in[p-2]:k_i'\ne n_i\}$.
Then $k_{i'}'\le n_{i'}-1$.
Since $\sum_{i=1}^{p}k_i'=\sum_{i=1}^{p-2}n_i+s+1$,
we have
$$
k_{p-1}'+k_p'=\sum_{i=1}^{p-2}n_i+s+1-\sum_{i=1}^{p-2}k_i'\ge s+2.
$$
Hence $k_{p-1}'\ge2.$
Define integers $k_1'',\ldots,k_p''$ by $k_i''=k_i'$ for $i\in[p]\setminus\{i',p-1\}$,
and $k_{i'}''=k_{i'}'+1,k_{p-1}''=k_{p-1}'-1$.
These integers are valid neighbor counts. Indeed,
$k_{i'}''=k_{i'}'+1\le n_{i'}, k_{p-1}''=k_{p-1}'-1\ge1$,
and
$1\le k_i''\le n_i$ for every $i\in[p]$.
Moreover,
$$
\sum_{i=1}^{p}k_i''=\sum_{i=1}^{p}k_i'=\sum_{i=1}^{p-2}n_i+s+1.
$$

Let $G'$ be the graph obtained from $T_{n-1,p}$ by adding a new vertex joined to exactly $k_i''$ vertices in $V_i$ for each $i\in[p]$.
 Then
$$
e(G')=e\big(T_{n-1,p}\big)+\sum_{i=1}^{p}k_i''=e\big(T_{n-1,p}\big)+\sum_{i=1}^{p}k_i'=e\big(Y_{n,p,s}\big).
$$
Moreover, since every $k_i''$ is positive, $G'$ contains a copy of $K_{p+1}$, and hence is non-$p$-partite.
Since the sequence $k_1',k_2',\ldots,k_p'$ is non-increasing and $i'\le p-2$,
we have $k_{i'}'\ge k_{p-1}'$.
Therefore
$$
(k_{i'}'+1)(k_{p-1}'-1)=k_{i'}'k_{p-1}'-(k_{i'}'-k_{p-1}'+1)<k_{i'}'k_{p-1}'.
$$
Consequently,
$$
\#K_{p+1}(G')=k_1''k_2''\cdots k_p''<k_1'k_2'\cdots k_p'=\#K_{p+1}(G),
$$
contradicting the extremal choice of $G$.
Hence $k_i'=n_i$ for every $i\in[p-2]$.

It remains to determine $k_{p-1}'$ and $k_p'$.
 From $\sum_{i=1}^{p}k_i'=\sum_{i=1}^{p-2}n_i+s+1$ and $k_i'=n_i$ for every $i\in[p-2]$,
we have $k_{p-1}'+k_p'=s+1$.
We claim that $k_p'=1$.
Suppose not. Since $k_p'\ge1$, this means $k_p'\ge2$.
As $k_{p-1}'\ge k_p',$ both $k_{p-1}'$ and $k_p'$ are at least $2$.
Therefore
$$
k_{p-1}'k_p'-s=k_{p-1}'k_p'-(k_{p-1}'+k_p'-1)=(k_{p-1}'-1)(k_p'-1)>0.
$$
Thus $k_{p-1}'k_p'>s$.
Since $k_i'=n_i$ for every $i\in[p-2],$ it follows that
$$
\#K_{p+1}(G)=k_1'k_2'\cdots k_p'=\Big(\prod_{i=1}^{p-2}n_i\Big)k_{p-1}'k_p'>\Big(\prod_{i=1}^{p-2}n_i\Big)s=\#K_{p+1}\big(Y_{n,p,s}\big),
$$
contradicting the extremal choice of $G$.
Hence $k_p'=1,$ and therefore $k_{p-1}'=s$.
This proves the claim.
\end{proof}

By Claim \ref{CLAIM3.11}, we have
$$
\#K_{p+1}(G)=k_1k_2\cdots k_p=k_1'k_2'\cdots k_p'=\Big(\prod_{i=1}^{p-2}n_i\Big)s=\#K_{p+1}\big(Y_{n,p,s}\big).
$$
Therefore, for both case $p=2$ and $p\geq 3$, we conclude that
$\#K_{p+1}(G)\geq \#K_{p+1}\big(Y_{n,p,s}\big).$
This completes the proof of Theorem \ref{thm1.1}.
\end{proof}


\begin{thebibliography}{99}
\setlength{\itemsep}{0pt}


\bibitem {Amin2013}
K. Amin, J. Faudree, R.J. Gould, E. Sidorowicz,
On the non-$(p-1)$-partite $K_p$-free graphs,
\emph{Discuss. Math. Graph Theory} \textbf{33} (1) (2013) 9--23.


\bibitem {Brouwer1981}
A. E. Brouwer, Some lotto numbers from an extension of Tur\'{a}n's theorem, Afdeling Zuivere Wiskunde, Vol. 152, Mathematisch Centrum, Amsterdam, 1981.

\bibitem {Erdos19626}
P. Erd\H{o}s, On a theorem of Rademacher-Tur\'{a}n,
\emph{Ill. J. Math.} \textbf{6} (1962) 122--127.

\bibitem {Erdos19627}
P. Erd\H{o}s, On the number of complete subgraphs contained in certain graphs,
\emph{Magyar Tud. Akad. Mat. Kut. Int. K{\"o}zl.} \textbf{7} (1962) 459--464.

\bibitem {Erdos1964}
P. Erd\H{o}s, On the number of triangles contained in certain graphs,
\emph{Canad. Math. Bull.} \textbf{7} (1964) 53--56.

\bibitem {Erdos-1967}
P. Erd\H{o}s,  Some recent results on extremal problems in graph theory (Results),
in: Theory of Graphs (International Symposium Rome, 1966), Gordon and
Breach, New York, Dunod, Paris, (1967), pp. 117--123.

\bibitem {Erdos-1968}
P. Erd\H{o}s, On some new inequalities concerning extremal properties of
graphs, in: Theory of Graphs (Proceedings of the Colloquium, Tihany, 1966),
Academic Press, New York, (1968), pp. 77--81.

\bibitem {Fisher1989}
D.C. Fisher, Lower bounds on the number of triangles in a graph,
\emph{J. Graph Theory} \textbf{13} (1989) 505--512.

\bibitem {Kang2005}
M. Kang, O. Pikhurko, Maximum $K_{r+1}$-free graphs which are not $r$-partite,
\emph{Mat. Stud.} \textbf{24} (2005), 12--20.


\bibitem{Komlos1996}
J. Koml\'{o}s, M. Simonovits, Szemer\'{e}di's regularity lemma and its applications in graph theory,
in: D. Mikl\'{o}s, V. S\'{o}s, T. Sz\H{o}nyi (Eds.), Combinatorics, Paul Erd\H{o}s is eighty, Vol. 2 (Keszthely, 1993), 295--352,
Bolyai Soc. Math. Stud., 2, J\'{a}nos Bolyai Math. Soc., Budapest, 1996.


\bibitem{LFP}
Y. T. Li, L. H. Feng, Y. J. Peng, Spectral supersaturation: Triangles and bowties, \emph{European J.
Combin.} \textbf{128} (2025), No. 104171.

\bibitem{LLP}
Y. T. Li, L. Lu, Y. J. Peng, A spectral Erd\H{o}s--Rademacher theorem,\emph{ Adv. in Appl. Math.}
\textbf{158} (2024), No. 102720.

\bibitem{LP2023}
Y. T. Li, Y. J. Peng, Refinement on spectral Tur\'{a}n's theorem,\emph{ SIAM J. Discrete Math.}
\textbf{37} (2023), no. 4, 2462–2485.


\bibitem {LM1975}
L. Lov\'{a}sz, M. Simonovits, On the number of complete subgraphs of a graph,
\emph{in: Proc. of Fifth British Comb. Conf. Aberdeen} (1975) 431--441.

\bibitem {LM1983}
L. Lov\'{a}sz, M. Simonovits, On the number of complete subgraphs of a graph II,
\emph{in: Studies in Pure Math, Birkh\"{a}user (to the Memory of Paul Tur\'{a}n)} (1983) 459--495.

\bibitem{LO}
H. Liu, O. Pikhurko, K. Staden, The exact minimum number of triangles in  graphs of given order and size,
\emph{Forum Math. Pi} \textbf{8} (2020), e8, 144 pp.


\bibitem{Mubayi2010}
D. Mubayi, Counting substructures I: Color critical graphs,
\emph{Adv. Math.} \textbf{225} (2010) 2731--2740.

\bibitem{MaYuan2025}
J. Ma, L.-T. Yuan, Supersaturation beyond color-critical graphs,
\emph{Combinatorica} \textbf{45} (2025), no. 2, Paper No. 18, 40 pp.

\bibitem{Moon1965}
J. W. Moon, On the number of complete subgraphs of a graph, \emph{Canad. Math. Bull.} \textbf{8} (1965) 831--834.

\bibitem {Mantel1907}
W. Mantel, Problem 28: Solution by H. Gouwentak, W. Mantel, J. Teixeira de Mattes, F.
Schuh and W. A. Wythoff, \emph{Wiskundige Opgaven} \textbf{10} (1907) 60--61.

\bibitem {Nikiforov2011}
V. Nikiforov, The number of cliques in graphs of given order and size,
\emph{Trans. Amer. Math. Soc.} \textbf{363} (2011) 1599--1618.

\bibitem{PY2017}
O. Pikhurko, Z. B. Yilma, Supersaturation problem for color-critical graphs,
\emph{J. Combin. Theory Ser. B} \textbf{123} (2017) 148--185.

\bibitem{PA2017}O. Pikhurko, A. Razborov, Asymptotic structure of graphs with the minimum number of triangles,
\emph{ Combin. Probab. Comput.} \textbf{26} (2017) 138--160.

\bibitem{Reiher2016}
C. Reiher, The clique density theorem,
\emph{Ann. Math.} \textbf{184} (2016) 683--707.

\bibitem {Simonovits1966}
M. Simonovits, A method for solving extremal problems in graph theory, stability problems,
in: Theory of Graphs (Proc. Colloq., Tihany, 1966),
\emph{Academic Press, New York} (1968), pp. 279--319.

\bibitem{Tyomkyn2015}
M. Tyomkyn, A. J. Uzzell, Strong Tur\'{a}n stability,
\emph{Electron. J. Combin.} \textbf{22} (2015), no. 3, P3.9.

\bibitem {Turan}
P. Tur\'{a}n, On an extremal problem in graph theory,
\emph{Mat. Fiz. Lapok} \textbf{48} (1941) 436--452.




\end{thebibliography}
\end{document}